\documentclass[reqno,11pt]{article}
\usepackage{a4wide}
\usepackage{color,eucal,enumerate,mathrsfs, amsthm}
\usepackage[normalem]{ulem}
\usepackage{amsmath,amssymb,epsfig,bbm}
\numberwithin{equation}{section}
\usepackage{tikz}

\usepackage[pdfborder={0 0 0}]{hyperref}  

\usepackage[latin1]{inputenc}
\usepackage[english]{babel}

\newcommand{\N}{\mathbb{N}}

\newcommand{\R}{\mathbb{R}}

\newcommand{\mm}{{\mbox{\boldmath$m$}}}

\newcommand{\sfd}{{\sf d}}

\newcommand{\Kliminf}{K\kern-3pt-\kern-2pt\mathop{\rm lim\,inf}\limits}  
\newcommand{\supp}{\mathop{\rm supp}\nolimits}   
\newcommand{\Lip}{\mathop{\rm Lip}\nolimits}          
\renewcommand{\d}{{\mathrm d}}

\newcommand{\restr}[1]{\lower3pt\hbox{$|_{#1}$}}
\newcommand{\la}{\left<}                  
\newcommand{\ra}{\right>}
\newcommand{\eps}{\varepsilon}  
\newcommand{\nchi}{{\raise.3ex\hbox{$\chi$}}}
\newcommand{\weakto}{\rightharpoonup}

\newcommand{\lims}{\varlimsup}

\newcommand{\fr}{\penalty-20\null\hfill$\blacksquare$}                      

\renewcommand{\mm}{\mathfrak m}                                

\renewenvironment{proof}{\removelastskip\par\medskip   
\noindent{\em proof} \rm}{\penalty-20\null\hfill$\square$\par\medbreak}

\newtheorem{theorem}{Theorem}[section]

\newtheorem{corollary}[theorem]{Corollary}
\newtheorem{lemma}[theorem]{Lemma}
\newtheorem{proposition}[theorem]{Proposition}

\newtheorem{definition}[theorem]{Definition}

\newtheorem{remark}[theorem]{Remark}

\newcommand{\RCD}{{\sf RCD}}

\newcommand{\lip}{{\rm lip}}

\newcommand{\ud}{\underline{\d}}
\newcommand{\LIP}{{\rm LIP}}
\renewcommand{\P}{{\sf Pr}}

\title{Behaviour of the reference measure on $\RCD$ spaces under charts}
\begin{document}

\author{
   Nicola Gigli\thanks{SISSA. email: \textsf{ngigli@sissa.it}} \ and Enrico Pasqualetto\thanks{SISSA. email: \textsf{epasqual@sissa.it}}
   }

\maketitle

\begin{abstract}

Mondino and Naber recently proved that finite dimensional $\RCD$ spaces are rectifiable. 

Here we show that the push-forward of the reference measure under the charts built by them is absolutely continuous with respect to the Lebesgue measure. This result, read in conjunction with another recent work of us, has relevant implications on the structure of tangent spaces to $\RCD$ spaces.

A key tool that we use is a recent paper by De Philippis-Rindler about the structure of measures on the Euclidean space.
\end{abstract}


\tableofcontents

\section{Introduction}
This paper is about the structure of charts in finite dimensional spaces with Ricci curvature bounded from below, $\RCD$ spaces in short. Our starting point is a result by Mondino-Naber \cite{Mondino-Naber14},  which can be roughly stated as:
\begin{quote}
Let $(X,\sfd,\mm)$ be a $\RCD^*(K,N)$ space and $\eps>0$. Then $\mm$-a.e.\ $X$ can be partitioned into a countable number of Borel subsets $(U_i)$, each $(1+\eps)$-biLipschitz to some subset of $\R^{n_i}$, where $n_i\leq N$ for every $i$.
\end{quote}
We shall provide the rigorous statement in Theorem \ref{thm:MN}.

In \cite{Mondino-Naber14}, the behaviour of the reference measure $\mm$ under the coordinate charts  is not studied. However, both for theoretical purposes ($\RCD$ spaces are metric \emph{measure} spaces, after all) and for practical ones (see the discussion below) it would be interesting to know the relation between $\mm$, the charts and the Lebesgue measure. This is the scope of this note, our main result being, again informally:
\begin{quote}
Let $(X,\sfd,\mm)$ be a $\RCD^*(K,N)$ space, $\eps>0$ sufficiently small and $(U_i,\varphi_i)$ the partition given by Mondino-Naber's theorem and the associated coordinate charts. Then 
\begin{equation}
\label{eq:thesis}
(\varphi_i)_*(\mm\restr{U_i})\ll\mathcal L^{n_i}\qquad\forall i.
\end{equation}
\end{quote}
See Theorem \ref{thm:main} for the precise statement and notice that our result is equivalent to the fact that  the restriction of $\mm$ to $U_i$ is absolutely continuous w.r.t.\ the $n_i$-dimensional Hausdorff measure.

We remark that in the case of Ricci-limit spaces (=mGH limits of Riemannian manifolds with uniform bound from below on the Ricci curvature and from above on the dimension), the analogous of our result  was already known from the work of Cheeger-Colding \cite{Cheeger-Colding97II}. However, the technique used in \cite{Cheeger-Colding97II} is not applicable to our setting, the problem being that in \cite{Cheeger-Colding97II} the spaces considered are limits of manifolds equipped with the volume measure, a  fact leading to some cancellations which are not present in the weighted case. Specifically, the key Lemma 1.14 in \cite{Cheeger-Colding97II} does not hold on weighted Riemannian manifolds, and a fortiori does not hold on $\RCD$ spaces.

\bigskip

Our argument, instead, uses as key tool the following recent result by De Philippis-Rindler \cite{DPR}:
\begin{theorem}\label{thm:DPR}
Let $T_1=\overrightarrow {T_1}\|T_1\|,\ldots,T_d=\overrightarrow{ T_d}\|T_d\|$ be one dimensional normal currents in $\R^d$ and $\mu$ a Radon measure on $\R^d$. Assume that:
\begin{itemize}
\item[i)] $\mu\ll \|T_i\|$ for $i=1,\ldots,d$,
\item[ii)] for $\mu$-a.e.\ $x$ the vectors $\overrightarrow {T_1}(x),\ldots,\overrightarrow {T_d}(x)$ are linearly independent.
\end{itemize}
Then $\mu$ is absolutely continuous w.r.t.\ the Lebesgue measure on $\R^d$.
\end{theorem}
We remark that such statement is only one of the several consequences of the main, beautiful, result in \cite{DPR}.

Our proof combines Theorem \ref{thm:DPR}, the construction by Mondino-Naber and the Laplacian comparison estimates obtained by the first author in \cite{Gigli12}  along the following lines:
\begin{itemize}
\item[i)] The typical  chart $\varphi$ in Mondino-Naber paper has coordinates which are distance functions from well chosen points, say $x_{1},\ldots,x_n$, and is $(1+\eps)$-biLipschitz on a set which we shall call $U$
\item[ii)] Assuming $\eps$ sufficiently small, it is not hard to see that the vector fields $v_i:=\nabla\sfd(\cdot,x_i)$ are independent on $U$
\item[iii)] The fact that the distance function has measure-valued Laplacian, grants that the $v_i$'s have measure valued divergence
\item[iv)] The differential of $\varphi$ sends the $v_i$'s to vector fields $u_i$  on $\R^n$ and with some algebraic manipulations one can see that ${\rm div}(u_i\varphi_*\mm)$ is still a measure
\item[v)] The fact that $\varphi\restr U$ is biLipschitz gives that the $u_i$'s are independent on $\varphi(U)$ (stated as such, this is not really correct - the precise formulation requires a cut off and an approximation procedure, see the proof of Theorem \ref{thm:main})
\item[vi)] Since vector fields with measure valued divergence are particular cases of 1-dimensional currents, the conclusion comes from Theorem \ref{thm:DPR} applied to the currents $u_i\varphi_*\mm$ and the measure $\varphi_*(\mm\restr U)$.
\end{itemize}
There are a few things that need to be explained/defined in this line of thought: this work will be carried out in Section \ref{se:tools}, while Section \ref{se:main} contains the statement and proof of our main result. 

One feature of our argument, which is basically a consequence of  Theorem \ref{thm:DPR}, is that we can prove \eqref{eq:thesis} for the charts whose coordinates are distance functions, whereas in \cite{Cheeger-Colding97II} their harmonic approximation was used: the only structural property we need is that the Laplacian of the coordinates is a measure.

\bigskip

Let us briefly describe a main consequence of our result. In \cite{Gigli14}, following some ideas of Weaver \cite{Weaver01}, it has been proposed an abstract definition of tangent `bundle' to a metric measure space based on the properties of Sobolev functions. For smooth Riemannian manifolds, this general  notion can trivially be identified with the classical concept of tangent space and thus also with the more geometric concept of pointed-measured-Gromov-Hausdorff limit of rescaled spaces. On the other hand, for general `irregular' spaces the approach in \cite{Gigli14} has little to do with tangent spaces arising as pmGH-limits.

One is therefore lead to look for sufficient regularity conditions on the general metric measure space that ensure the equivalence of these two notions. This has been the scope of our  companion paper \cite{GP16}: there we proved that if a space can be covered with $(1+\eps)$-biLipschitz charts satisfying \eqref{eq:thesis}, then indeed such equivalence is in place. As discussed in \cite{GP16}, the main example of application of our result is the one of $\RCD^*(K,N)$ spaces, where the $(1+\eps)$-biLipschitz charts are given by Mondino-Naber and the absolute continuity property \eqref{eq:thesis} by this manuscript. This is relevant because it opens up the possibility of studying the `concrete and geometric' notion of tangent space as pmGH-limit via the `abstract and analytic' one proposed in  \cite{Gigli14}.

\bigskip

Finally we remark that other two independent recent papers (\cite{DPMR16} and \cite{MK16}) cover results overlapping with ours; let us briefly describe those  and the relations with ours. In \cite{DPR} it has been observed how combining the main results of \cite{DPR} and \cite{AM16} it is possible to deduce a converse of Rademacher theorem, namely that if $\mu$ is a measure on $\R^d$ such that every Lipschitz function is differentiable $\mu$-a.e., then necessarily $\mu\ll\mathcal L^d$. In \cite{DPMR16}, it has then been noticed how this latter result together with the characterization of measures on Lipschitz differentiability spaces obtained by Bate in \cite{Bate15}, implies the validity of Cheeger's conjecture on Lipschitz differentiability spaces, namely the analogous of our main theorem with `Mondino-Naber charts on $\RCD$ spaces' replaced by `charts in a Lipschitz differentiability space'. In \cite{MK16}, among other things, this line of thought has been pushed to obtain our very same theorem on $\RCD$ spaces: the added observation is that Cheeger's results in  \cite{Cheeger00} ensure that $\RCD$ spaces equipped with the Mondino-Naber charts are Lipschitz differentiability spaces.

Despite this overlapping, we believe that our approach has some independent interest: as discussed above, working with the added regularity of $\RCD$ spaces allows us to quickly conclude from Theorem \ref{thm:DPR}, without the need of using also the deep results in  \cite{Cheeger00}, \cite{Bate15} and \cite{AM16}.

\bigskip

\noindent{\bf Acknowledgment }

\noindent This research has been supported by the MIUR SIR-grant `Nonsmooth Differential Geometry' (RBSI147UG4).

\section{Technical tools}\label{se:tools}

\subsection{Some properties of the differential of a map between metric measure spaces}

To keep the presentation short, we assume the reader familiar with the language of $L^\infty$-modules developed in  \cite{Gigli14}. We shall only recall, without proof, those definitions and properties we need.

Let $(X,\sfd_X,\mm_X),(Y,\sfd_Y,\mm_Y)$ be two metric measure spaces and $\varphi:X\to Y$ a map of {\bf bounded compression}, i.e.\ so that $\varphi_*\mm_X\leq C\mm_Y$ for some $C>0$.

Given an $L^2(Y)$-normed module $M$, the pullback $\varphi^*M$, which is an $L^2(X)$-normed module, and the pullback map $\varphi^*:M\to \varphi^*M$, which is linear and continuous, are characterised up to unique isomorphism by the fact that
\[
\begin{split}
|\varphi^*v|&=|v|\circ\varphi,\quad\mm_X-a.e.\ \forall v\in M,\\
\{\varphi^*v&:\ v\in M\}\text{ generates the whole $\varphi^*M$}.
\end{split}
\]
Notice that if $M=L^2(Y)$, then $\varphi^*M= L^2(X)$ with $\varphi^*f= f\circ \varphi$.

Given an $L^2(Y)$-normed module $M$ and its dual $M^*$, there is a unique continuous $L^\infty(X)$-bilinear map from $\varphi^*M\times \varphi^*M^*$ to $L^1(X)$ such that
\begin{equation}
\label{eq:dualpair}
\varphi^*L(\varphi^*v)=L(v)\circ\varphi\qquad\forall v\in M,\ L\in M^*.
\end{equation}
Such duality pairing provides an isometric embedding of $\varphi^*M^*$ into the dual of $\varphi^*M$, but in general such embedding is not surjective. A sufficient condition for surjectivity is that $M^*$ is separable (this has to do with the Radon-Nikodym property of $M^*$).

In the special case in which $M=L^2(T^*Y)$ is the cotangent module of $Y$, we shall denote the pullback map by $\omega\mapsto [\varphi^*\omega]$, to distinguish it by the pullback of 1-forms whose definition we recall in a moment.

\bigskip

Now we assume that not only $\varphi:X\to Y$ is of bounded compression, but also that it is Lipschitz. Maps of this kind are called of \textbf{bounded deformation}. 

Recall that given a map $\varphi:X\to Y$  of bounded deformation, the map from $W^{1,2}(Y)$ to $W^{1,2}(X)$ sending $f$ to $f\circ\varphi$ is linear and continuous, and that it holds
\[
|\d(f\circ\varphi)|\leq \Lip(\varphi)|\d f|\circ\varphi.
\]
It can then be seen that there is a unique linear and continuous map $\varphi^*:L^2(T^*Y)\to L^2(T^*X)$, called pullback of 1-forms, such that
\begin{equation}
\label{eq:defpullb}
\begin{split}
\varphi^*\d f&=\d(f\circ\varphi)\qquad\forall f\in W^{1,2}(Y),\\
\varphi^*(g\omega)&=g\circ\varphi\varphi^*\omega\qquad\forall \omega\in L^2(T^*Y),\ g\in L^\infty(Y),
\end{split}
\end{equation}
and that it also satisfies
\begin{equation}
\label{eq:normpullb}
|\varphi^*\omega|\leq \Lip(\varphi)|\omega|\circ\varphi\qquad\mm_X-a.e..
\end{equation}

Recall that in this setting the tangent module is defined as the dual of the cotangent one; still, to keep consistency with the notation used in the smooth case, the duality pairing between $v\in L^2(TX)$ and $\omega\in L^2(T^*X)$ is denoted by $\omega(v)$.

The differential $\d\varphi$ of $\varphi$ is then defined as follows.
\begin{definition}[The differential of a map of bounded deformation]
Let $\varphi:X\to Y$ be of bounded deformation and assume that $L^2(TY)$ is separable. The differential $\d\varphi:L^2(TX)\to \varphi^*(L^2(TY))$ is the only linear and continuous map such that
\begin{equation}
\label{eq:defdiff}
[\varphi^*\omega](\d\varphi(v))=\varphi^*\omega(v)\qquad\forall \omega\in L^2(T^*Y),\ v\in L^2(TX).
\end{equation}
\end{definition}
The separability assumption on $L^2(TY)$ is needed because  \eqref{eq:defdiff} only defines an element of the dual of $\varphi^*L^2(T^*Y)$ which a priory might be larger than $\varphi^*L^2(TY)$ (recall the duality pairing \eqref{eq:dualpair}).

It turns out that $\d\varphi$ is also $L^\infty(X)$-linear and satisfies
\begin{equation}
\label{eq:bounddiff}
|\d\varphi(v)|\leq\Lip(\varphi)|v|\quad\mm_X-a.e.\qquad\forall v\in L^2(TX).
\end{equation}

\bigskip

Much like in the classical smooth setting, part of the necessity of calling into play the pullback module is due to the fact that $\varphi$ might be not injective, so that one cannot hope to define $\d\varphi(v)(y)\in T_yY$  as $\d\varphi_{\varphi^{-1}(y)}(v(\varphi^{-1}(y)))$, because $\varphi^{-1}(y)$ can contain more than one point. 

A way to assign to each vector field on $X$ a vector field on $Y$ via the differential of $\varphi$ is to, roughly said, take the average of $\d\varphi_{\varphi^{-1}(y)}(v(\varphi^{-1}(y)))$ among all the preimages of $y$. Rigorously, this is achieved by introducing a left inverse $\P_\varphi:\varphi^*M\to M$ of the pullback map $\varphi^*:M\to \varphi^*M$, as we discuss now. 

\bigskip

We shall assume from now on that $\varphi_*\mm_X=\mm_Y$. For $f\in L^p(X)$ non-negative we put
\[
\P_\varphi(f):=\frac{\d\varphi_*(f\mm_X)}{\d\mm_Y}
\]
and for general $f\in L^p(X)$ we put $\P_\varphi(f):=\P_\varphi(f^+)-\P_\varphi(f^-)$. By first checking the cases $p=1,\infty$ it is easy to verify that $\P_\varphi:L^p(X)\to L^p(Y)$ is linear, continuous and satisfies
\[
|\P_\varphi(f)|\leq \P_\varphi(|f|),\qquad\mm_Y-a.e..
\]
In the case of general modules, the map $\P_\varphi:\varphi^*M\to M$ can be characterized as the only linear and continuous map such that 
\begin{equation}
\label{eq:defpv}
\P_\varphi(f\varphi^*v)=\P_\varphi(f)v,\qquad\forall f\in L^\infty(X),\ v\in M,
\end{equation}
and it can be verified that the bound
\[
|\P_\varphi(V)|\leq \P_\varphi(|V|),\quad\mm_Y-a.e.\qquad\forall V\in \varphi^*M
\]
holds. Notice that, analogously to \eqref{eq:defpv}, it also holds
\begin{equation}
\label{eq:altrap}
g\P_\varphi(V)=\P_\varphi(g\circ\varphi V),\qquad\forall g\in L^\infty(Y),\ V\in\varphi^*M.
\end{equation}
Indeed, for given $g$ both sides of this identity are linear and continuous in $V$ and agree on those $V$'s of the form $f\varphi^*v$ for $f\in L^\infty(X)$, $v\in M$.

All these definitions and properties can be found in  \cite{Gigli14}. Now we turn to the main result of this section: we are interested in studying the map $v\mapsto \P_\varphi(\d\varphi(v))$ under the assumption that for some Borel $E\subset X$ the restriction of $\varphi$ to $E$ is invertible and with Lipschitz inverse. 

We shall use the following notation: for a given $L^\infty$-module $M$ and Borel set $E$ we shall denote by $M\restr E$ the set of those $v\in M$ which are concentrated on $E$, i.e.\ such that $\nchi_{E^c}v=0$.

We recall that $v_1,\ldots,v_n\in M$ are said independent on $E$ provided for any $f_1,\ldots,f_n\in L^\infty$ we have
\[
\nchi_E\sum_if_iv_i=0\qquad\Rightarrow\qquad \nchi_E f_i=0\quad\forall i.
\]

In the course of the proof we shall use the identity
\begin{equation}
\label{eq:idp}
\omega(\P_\varphi(V))=\P_\varphi([\varphi^*\omega](V))\qquad\forall \omega\in L^2(T^*Y),\ V\in \varphi^*L^2(TY),
\end{equation}
which can be easily proved  by noticing that for given $\omega\in L^2(T^*Y)$ the two sides define linear continuous maps from $\varphi^*L^2(TY)$ to $L^1(Y)$ which agree on $V$'s of the form $f\varphi^*v$ for $f\in L^\infty(X)$ and $v\in L^2(TY)$.

\begin{proposition}\label{prop:ind}
Let $\varphi:X\to Y$ be of bounded deformation, with $\varphi_*\mm_X=\mm_Y$ and assume that for some Borel set $E\subset X$ we have that $\varphi\restr E$ is injective with $(\varphi\restr E)^{-1}$ Lipschitz. Assume also that Lipschitz functions on $X$ are dense in $W^{1,2}(X)$.

Then the map
\[
L^2(TX)\restr E\ni v\quad\mapsto\quad\P_\varphi(\d\varphi(v))\in L^2(TY)
\]
is injective. 

In particular if $v_1,\ldots,v_n\in L^2(TX)$ are independent on $E$, then the vectors \linebreak $\P_\varphi(\d\varphi(\nchi_Ev_1)), \ldots,\P_\varphi(\d\varphi(\nchi_Ev_n))\in L^2(TY)$ are independent on $\{\P_\varphi(\nchi_E)>0\}\subset Y$.
\end{proposition}
\begin{proof} By inner regularity of $\mm_X$ we can, and will, assume that $E$ is compact.
The assumption that Lipschitz functions on $X$ are dense in $W^{1,2}(X)$ grants  that $\{\d f:f\in\LIP\cap W^{1,2}(X) \}$ is dense in  $\{\d f:f\in  W^{1,2}(X)\}$ w.r.t.\ the $L^2(T^*X)$ topology. Recalling that $L^2(T^*X)$ is generated by the differentials of functions in $W^{1,2}(X)$ we therefore deduce that
\begin{equation}
\label{eq:vdenso}
V:=\Big\{\nchi_E\sum_{i=1}^nh_i\d f_i:\ n\in\N,\ f_i\in \LIP\cap W^{1,2}(X), h_i\in L^\infty(X)\Big\}\qquad\text{ is dense in }L^2(T^*X)\restr{E}.
\end{equation}
Now let $f\in  \LIP\cap W^{1,2}(X)$, consider the Lipschitz function $f\circ(\varphi\restr E)^{-1}$ defined on $\varphi(E)$ and extend it to a Lipschitz function $g$ on $Y$ with bounded support. Then  $g\in W^{1,2}(Y)$ and $g\circ\varphi =f$ on $E$. This identity and the locality of the differential (see  \cite{Gigli14}) imply that $\nchi_E\d f=\nchi_E\d(g\circ\varphi)$ so that taking into account  the first in \eqref{eq:defpullb} we  have
\[
\nchi_E\d f=\nchi_E\d (g\circ\varphi)=\nchi_E\,\varphi^*\d g \quad \in\quad \nchi_E({\rm Im }\,\varphi^*).
\]
Since the second in \eqref{eq:defpullb} and the assumption about the invertibility of $\varphi\restr E$ ensure that $\nchi_E({\rm Im }\,\varphi^*)$ is closed under $L^\infty(Y)$-linear combinations, we deduce that $V\subset  \nchi_E({\rm Im }\,\varphi^*)$, which together with \eqref{eq:vdenso} implies
\begin{equation}
\label{eq:imdensa}
\nchi_E ({\rm Im }\,\varphi^*)\quad\text{ is dense in }L^2(T^*X)\restr{E}.
\end{equation}

Next, we claim that
\begin{equation}
\label{eq:p0}
f\in L^1(X)\restr E\text{ and } \P_\varphi(f)=0\qquad\Rightarrow\qquad f=0.
\end{equation}
This can be seen by letting $g\in L^1(Y)$ be defined as ${\rm sign}(f\circ\varphi\restr E^{-1})$ on $\varphi(E)$ and $0$ outside. Then it holds 
\[
0=\int g\P_\varphi(f)\,\d\mm_Y=\int g\d\varphi_*(f\mm_X)=\int g\circ\varphi f\,\d\mm_X=\int |f|\,\d\mm_X.
\]

The injectivity claim now follows noticing that for  $v\in L^2(TX)\restr E$ we have
\[
\begin{split}
\P_\varphi(\d\varphi (v))=0\qquad&\Leftrightarrow\qquad \omega\big(\P_\varphi(\d\varphi (v))\big)=0\qquad \forall \omega\in L^2(T^*Y)\\
\text{(by \eqref{eq:idp})}\qquad\qquad&\Leftrightarrow\qquad \P_\varphi\big([\varphi^*\omega](\d\varphi (v))\big)=0\qquad \forall \omega\in L^2(T^*Y)\\
\text{(by \eqref{eq:p0} and $v\in L^2(TX)\restr E$)}\qquad\qquad&\Leftrightarrow\qquad  [\varphi^*\omega](\d\varphi (v))=0\qquad \forall \omega\in L^2(T^*Y)\\
\text{(by \eqref{eq:defdiff})}\qquad\qquad&\Leftrightarrow\qquad  \varphi^*\omega(v)=0\qquad \forall \omega\in L^2(T^*Y)\\
\text{(by \eqref{eq:imdensa} and $v\in L^2(TX)\restr E$)}\qquad\qquad&\Leftrightarrow\qquad  v=0.
\end{split}
\]
For the last claim simply observe that for $f_i\in L^\infty(Y)$ we have
\[
\begin{split}
\sum_if_i\P_\varphi(\d\varphi(\nchi_Ev_i))\stackrel{\eqref{eq:altrap}}=\sum_i\P_\varphi(f_i\circ\varphi\,\d\varphi(\nchi_Ev_i))=\P_\varphi\Big(\d\varphi\big(\nchi_E\sum_if_i\circ\varphi \,v_i\big)\Big)
\end{split}
\]
and therefore 
\[
\begin{split}
\sum_if_i&\P_\varphi(\d\varphi(\nchi_Ev_i))=0\quad \\
&\Leftrightarrow\quad \nchi_E\sum_if_i\circ\varphi \,v_i=0\qquad\text{ by the injectivity just proved}\\
&\Leftrightarrow\quad f_i\circ \varphi =0\quad\mm_X\restr E-a.e.\quad \forall i\qquad\text{ by the independence of $(v_i)$ on $E$ }\\
&\Leftrightarrow\quad f_i =0\quad\varphi_*(\mm_X\restr E)-a.e.\quad\forall i\\
&\Leftrightarrow\quad f_i =0\quad\mm_Y-a.e.\ on\ \{\P_\varphi(\nchi_E)>0\}\quad\forall i,
\end{split}
\]
which is the thesis.
\end{proof}
\begin{remark}{\rm
Given inequality \eqref{eq:bounddiff} and taking into account the weighting given by the operator $\P_\varphi$, one might expect that under the assumptions of the previous proposition, not only the stated injectivity holds, but actually that the quantitative bound
\[
|\P_\varphi(\d\varphi(v))|\leq \Lip\big((\varphi\restr E)^{-1}\big)\,\P_\varphi(\nchi_E)\circ\varphi\,|v|
\]
holds. Yet, this is not clear: the problem is that we don't know whether $(\varphi\restr E)^{-1}$ can be extended to a map of bounded deformation.
}\fr\end{remark}

\subsection{Measure valued divergence}
Here we discuss the notion of measure valued divergence, mimicking the one of measure valued Laplacian given in  \cite{Gigli12}. For an earlier approach to this sort of definition see \cite{Gigli-Mondino12}. The definition and results presented are valid on arbitrary metric measure spaces $(X,\sfd,\mm)$ so that $(X,\sfd)$ is proper (although this can in fact be relaxed) and $\mm$ a non-negative Radon measure. In particular, the measure valued divergence is always a linear operator (unlike the Laplacian, whose linearity requires the infinitesimal Hilbertianity assumption).

\begin{definition}[Measure valued divergence]\label{def:mdiv}
Let $\Omega\subset X$ be open and $v\in L^2(TX)$. We say that $v$ has measure valued divergence in $\Omega$, and write $v\in D(\bold{div}_\mm,\Omega)$ if there exists a Radon measure $\mu$ on $\Omega$ such that 
\[
\int \d f(v)\,\d\mm=-\int f\,\d\mu
\]
for every Lipschitz function $f$ with  support compact and contained in $\Omega$. In this case the measure $\mu$, which is clearly unique, will be denoted  $\bold{div}_\mm\restr\Omega(v)$.

In the case $\Omega=X$ we shall simply write $D(\bold{div}_\mm)$ and $\bold{div}_\mm(v)$.
\end{definition}
We have the following two simple basic calculus rules for the divergence, both being consequences of the Leibniz rule for the differential.
\begin{proposition}[Leibniz rule]\label{prop:leibdiv}
Let $v\in D(\bold{div}_\mm,\Omega)$ and $g:X\to\R$ Lipschitz and bounded. Then $gv\in D(\bold{div}_\mm,\Omega)$ and
\[
\bold{div}_\mm\restr{\Omega}(gv)=g\bold{div}_\mm\restr{\Omega}(v)+\d g(v)\mm\restr{\Omega}.
\]
\end{proposition}
\begin{proof}
Observe that for $f:X\to\R$ Lipschitz with support compact and contained in $\Omega$ it holds
\[
-\int f\,\d\big(g\bold{div}_\mm\restr{\Omega}(v)+\d g(v)\mm\restr{\Omega}\big)=\int \d(fg)(v)-f\d g(v)\,\d\mm=\int \d f(gv)\,\d\mm,
\]
which is the thesis.
\end{proof}

\begin{proposition}[Locality]\label{prop:localdiv}
Let $\Omega_1,\Omega_2\subset X$ open and $v\in D(\bold{div}_\mm,\Omega_1)\cap D(\bold{div}_\mm,\Omega_1)$. Then 
\begin{equation}
\label{eq:local}
\big(\bold{div}_\mm\restr{\Omega_1}(v)\big)\restr{\Omega_1\cap\Omega_2}=\big(\bold{div}_\mm\restr{\Omega_2}(v)\big)\restr{\Omega_1\cap\Omega_2},
\end{equation}
$v\in D(\bold{div}_\mm,\Omega_1\cup\Omega_2)$ and it holds
\begin{equation}
\label{eq:union}
\big(\bold{div}_\mm\restr{\Omega_1\cup\Omega_2}(v)\big)\restr{\Omega_i}=\bold{div}_\mm\restr{\Omega_i}(v)\qquad i=1,2.
\end{equation}
\end{proposition}
\begin{proof}
To prove \eqref{eq:local} it is sufficient to consider Lipschitz functions with support in $\Omega_1\cap\Omega_2$, which are dense in $C_c(\Omega_1\cap\Omega_2)$, in the definition of  $\bold{div}_\mm\restr{\Omega_1}(v),\bold{div}_\mm\restr{\Omega_2}(v)$. For \eqref{eq:union} let $f:X\to\R$ be Lipschitz with support compact and contained in $\Omega:=\Omega_1\cup\Omega_2$ and $\nchi_1,\nchi_2:X\to[0,1]$ a Lipschitz partition of the unit of the space $\supp(f)$ subordinate to the cover $\{\Omega_1,\Omega_2\}$. Then letting $\mu$ be the measure defined by \eqref{eq:union} we have that
\[
\begin{split}
-\int f\,\d\mu&=-\int f\nchi_1\,\d\bold{div}_\mm\restr{\Omega_1}(v)-\int f\nchi_2\,\d \bold{div}_\mm\restr{\Omega_2}(v)\\
&=\int \big(\d(f\nchi_1)+\d(f\nchi_2)\big)(v)\,\d\mm=\int \d f(v)\,\d\mm,
\end{split}
\]
having used the fact that $\d(\nchi_1+\nchi_2)=\d 1=0$.
\end{proof}
Finally, we point out how the measure valued divergence is transformed under maps of bounded deformation:
\begin{proposition}\label{prop:pfdiv}
Let $\varphi:X\to Y$ be proper (=preimage of compact is compact) and of bounded deformation and such that $\mm_Y=\varphi_*\mm_X$. Then for any $v\in L^2(TX)$ and $f\in W^{1,2}(Y)$ we have
\[
\int \d f({\sf Pr}_\varphi(\d\varphi(v)))\,\d\mm_Y=\int \d(f\circ\varphi)(v)\,\d\mm_X.
\]
In particular, if $v\in D({\bold{ div}_{\mm_X}})$, then ${\sf Pr}_\varphi(\d\varphi(v))\in D({\bold{ div}}_{\mm_Y})$ and
\[
\bold{div}_{\mm_Y}\big({\sf Pr}_\varphi(\d\varphi(v))\big)=\varphi_*\big(\bold{div}_{\mm_X}(v)\big).
\]

\end{proposition}
\begin{proof} Pick $f:X\to\R$ Lipschitz with compact support. Recalling \eqref{eq:idp} and the definition of $\d\varphi(v)$ we have
\[
 \d f({\sf Pr}_\varphi(\d\varphi(v)))={\sf Pr}_\varphi([\varphi^*\d f](\d\varphi(v)))={\sf Pr}_\varphi(\d(f\circ\varphi)(v)).
\]
Integrating w.r.t.\ $\mm_Y=\varphi_*\mm_X$ and using the trivial identity $\int\P_\varphi(g)\,\d\mm_Y=\int g\,\d\mm_X$ valid for any $g\in L^1(X)$ we deduce
\[
\int \d f({\sf Pr}_\varphi(\d\varphi(v)))\,\d\varphi_*\mm_X=\int \d(f\circ\varphi)(v)\,\d\mm_X=-\int f\circ\varphi\,\d{\bold{div}_{\mm_X}}(v)=-\int f\,\d\varphi_*{\bold{div}_{\mm_X}}(v),
\]
which, by the arbitrariness of $f$, is the thesis.
\end{proof}

\subsection{About (co)vector fields on weighted $\R^d$}

Let us consider the Euclidean space $\R^d$ equipped with a non-negative Radon measure $\mu$. Here we have at least two ways of speaking about, say, $L^2(\mu)$ vector fields: one is simply to consider the space $L^2(\R^d,\R^d;\mu)$ of $L^2(\mu)$-maps from $\R^d$ to itself, the other is via the abstract notion of tangent module, which we shall denote as $L^2_\mu(T\R^d)$. 

Such two spaces are in general different, as can be seen by considering the case of $\mu$ being a Dirac delta: in this case $L^2(\R^d,\R^d;\mu)$ has dimension $d$ while $L^2_\mu(T\R^d)$ reduces to the 0 space. Aim of this section is to show that $L^2_\mu(T\R^d)$ always canonically and isometrically embeds in $L^2(\R^d,\R^d;\mu)$. This is useful because once we have such `concrete' representations of vector fields in $L^2_\mu(T\R^d)$,  we will be able to canonically associate 1-currents to them. As we shall see at the end of the section, the fact that this current is normal is essentially equivalent to the fact that the original vector field had measure valued divergence in the sense of Definition \ref{def:mdiv}.

A word on notation: to distinguish between the classically defined differential and the one coming from the theory of modules, we shall denote the former by $\ud f$, while keeping $\d f$ for the latter. More generally, elements of $L^2(\R^d,\R^d;\mu)$ or $L^2(\R^d,(\R^d)^*;\mu)$ will typically be underlined, while those of $L^2_\mu(T\R^d)$, $L^2_\mu(T^*\R^d)$ will be not.

\bigskip

Consider the set $V\subset  L^2(\R^d,(\R^d)^*;\mu)$ defined by
\[
V:=\Big\{\sum_{i=1}^n\nchi_{A_i}\ud f_i\ :\ n\in\N,\ (A_i)\text{ disjoint Borel subsets of }\R^d,\ f_i\in C^1_c(\R^d)\Big\}
\]
and define $P:V\to L^2_\mu(T^*\R^d)$ by
\[
P\Big(\sum_{i=1}^n\nchi_{A_i}\ud f_i\Big):=\sum_{i=1}^n\nchi_{A_i}\d f_i.
\]
We have the following simple result:
\begin{proposition}
The map $P$ is well defined and uniquely extends to a linear continuous map, still denoted by $P$, from $L^2(\R^d,(\R^d)^*;\mu)$ to $L^2_\mu(T^*\R^d)$. Such extension is a $L^\infty$-module morphism and satisfies
\begin{equation}
\label{eq:leq}
|P(\underline\omega)|\leq |\underline\omega|,\quad\mu-a.e.\qquad\forall \underline\omega\in L^2(\R^d,(\R^d)^*;\mu).
\end{equation}
\end{proposition}
\begin{proof} The trivial inequality 
\[
|\d f|\leq |\ud f|=\lip(f)\quad\mu-a.e.,
\]
valid for every $f\in C^1_c(\R^d)$ (see e.g. \cite{AmbrosioGigliSavare11-3}) grants that $P$ is well defined and that the bound \eqref{eq:leq} for $\underline\omega\in V$ holds. Such bound also ensures that $P$ is continuous and since, as is obvious,  $V$ is a dense vector subspace of $L^2(\R^d,(\R^d)^*;\mu)$, we get existence and uniqueness of the continuous extension, which is also clearly linear. The fact that such extension is a $L^\infty$-modules morphism can be checked by first noticing that by definition $P$ behaves properly w.r.t.\ multiplication by simple functions and then arguing by approximation. 
\end{proof}

By duality we can then define a map $\iota:L^2_\mu(T\R^d)\to L^2(\R^d,\R^d;\mu)\sim L^2(\R^d,(\R^d)^*;\mu)^*$ by:
\[
\underline\omega(\iota(v)):=P(\underline\omega)(v),\qquad\forall v\in L^2_\mu(T\R^d),\ \underline\omega\in L^2(\R^d,(\R^d)^*;\mu).
\]
Then the bound \eqref{eq:leq} directly gives
\begin{equation}
\label{eq:leq2}
|\iota(v)|\leq |v|,\quad\mu-a.e.\qquad\forall v\in L^2_\mu(T\R^d).
\end{equation}
We want to prove that equality holds here, i.e. that $\iota$ is actually an isometric embedding. We shall obtain this by proving that $P$ is a quotient map, more specifically that it is  surjective and such that
\begin{equation}
\label{eq:pquot}
|\omega|=\min_{\underline\omega\in P^{-1}(\omega)}|\underline\omega|,\quad\mu-a.e.\qquad\forall \omega\in L^2_\mu(T^*\R^d).
\end{equation}
We shall need the following lemma about the structure of Sobolev spaces over weighted $\R^d$:
\begin{lemma}\label{le:apc1}
Let $f\in W^{1,2}(\R^d,\sfd_{\rm Eucl},\mu)$. Then there exists a sequence $(f_n)\subset C^1_c(\R^d)$ converging to $f$ in $L^2(\mu)$ such that $|\ud f_n|\to |\d f|$ in $L^2(\mu)$.
\end{lemma}
\begin{proof}
It is known from  \cite{AmbrosioGigliSavare11-3} that for any $f\in W^{1,2}(\R^d,\sfd_{\rm Eucl},\mu)$ there exists a sequence $(f_n)$  of Lipschitz and compactly supported functions on $\R^d$ such that $(f_n) ,\lip_a(f_n)$ converge to $f,|\d f|$ respectively in $L^2(\mu)$. Here $\lip_a(g)$ is the asymptotic Lipschitz constant defined as
\[
\lip_a(g)(x):=\lims_{y,z\to x}\frac{|g(y)-g(z)|}{|y-z|}=\lim_{r\downarrow 0}\Lip(g\restr{B_r(x)})=\inf_{r> 0}\Lip(g\restr{B_r(x)}).
\]
Therefore by a diagonalization argument to conclude it is sufficient to show that given $f$ Lipschitz with compact support we can find $(f_n)\subset C^1_c(\R^d)$ uniformly Lipschitz, converging to $f$ in $L^2(\mu)$ and such that 
\begin{equation}
\label{eq:c1}
\lims_{n}|\ud f_n|(x)\leq \lip_af(x),\qquad\forall x\in\R^d.
\end{equation}
To this aim we simply define $f_n:=f*\rho_n\in C^1_c(\R^d)$, where $(\rho_n)$ is a standard family of mollifiers such that $\supp(\rho_n)\subset B_{1/n}(0)$. It is trivial that $(f_n)$ converges to $f$ in $L^2(\mu)$, that the family is equiLipschitz (the global Lipschitz constant being bounded by that of $f$) and that for $\frac1n<r$ it holds 
\[
|\ud f_n|(x)\leq \Lip(f\restr{B_r(x)}),\qquad\forall x\in\R^d.
\]
Letting first $n\to\infty$ and then $r\downarrow0$ we get \eqref{eq:c1} and the conclusion.
\end{proof}
Thanks to this approximation result, we obtain the following:
\begin{proposition}\label{thm:isomemb}
The map $P:L^2(\R^d,(\R^d)^*;\mu)\to L^2_\mu(T^*\R^d)$ is surjective and satisfies \eqref{eq:pquot} and the map $\iota:L^2_\mu(T\R^d)\to L^2(\R^d,\R^d;\mu)$ is an $L^\infty$-module morphism preserving the pointwise  norm, i.e.
\begin{equation}
\label{eq:iso}
|\iota(v)|= |v|,\quad\mu-a.e.\qquad\forall v\in L^2_\mu(T\R^d).
\end{equation}
In particular, if $v_1,\ldots,v_n\in L^2_\mu(T\R^d)$ are independent on $E\subset \R^d$, then $\iota(v_1)(x),\ldots,\iota(v_n)(x)\in\R^d$ are independent for $\mu$-a.e.\ $x\in E$.
\end{proposition}
\begin{proof}
We start showing that for $f\in W^{1,2}(\R^d,\sfd_{\rm Eucl},\mu)$ we have that $\d f$ belongs to the range of $P$. To this aim, let $(f_n)\subset C^1_c(\R^d)$ be as in Lemma \ref{le:apc1} and notice that such lemma grants that $(\ud f_n)$ is a bounded sequence in $L^2(\R^d,(\R^d)^*;\mu)$. Being such space reflexive,  up to pass to a non-relabeled subsequence we can assume that $\ud f_n\weakto \underline\omega$ for some $\underline\omega\in L^2(\R^d,(\R^d)^*;\mu)$.

Being $P$ linear and continuous we know that $\d f_n=P(\ud f_n)\weakto P(\underline\omega)$ in $L^2_\mu(T^*\R^d)$ and this fact together with the  closure of the differential (see \cite{Gigli14}) grants that $\d f=P(\underline\omega)$, thus giving the claim.

Lemma \ref{le:apc1} grants that $|\ud f_n|\to |\d f|$ in $L^2(\mu)$ and it is easy to check that this grants $|\underline\omega|\leq |\d f|$ $\mu$-a.e., so that we have 
\begin{equation}
\label{eq:chain}
|\d f|=|P(\underline\omega)|\leq |\underline \omega|\leq |\d f|,\qquad\mu-a.e.,
\end{equation}
which forces the equalities and thus  shows that \eqref{eq:pquot} holds for $\omega:=\d f$. 

Since $P$ is a $L^\infty$-module morphism, we deduce that any $\omega$ of the form $\sum_{i=1}^n\nchi_{A_i}\d f_i$ is in the image of $P$ and that for such $\omega$'s the identity \eqref{eq:pquot} holds. To conclude for the first part of the statement, pick $\omega\in L^2_\mu(T^*\R^d)$ and a sequence $(\omega_n)\subset  L^2_\mu(T^*\R^d)$ of finite $L^\infty$-linear combinations of differentials. By what we just proved there are $\underline \omega_n\in P^{-1}(\omega_n)$ realizing the equality in \eqref{eq:pquot}. In particular, $(\underline \omega_n)$ is a bounded sequence in $L^2(\R^d,(\R^d)^*;\mu)$ and thus up to pass to a subsequence it weakly converges to some $\underline\omega$. It is clear that $P(\underline\omega)=\omega$ and, arguing as before, that $|\omega|=|\underline\omega|$ $\mu$-a.e..

We turn to the second part of the statement. The fact that $\iota$ is a $L^\infty$-module morphism is obvious by definition. Now pick $v\in L^2_\mu(T\R^d)$, $\eps>0$ and find  $\omega\in L^2_\mu(T^*\R^d)$ with $\||\omega|\|_{L^2(\mu)}=1$ and $\int \omega(v)\,\d\mu\geq \|v\|_{L^2(\mu)}-\eps$. Then use what previously proved to find $\underline\omega\in L^2(\R^d,(\R^d)^*;\mu)$ with $|\underline\omega|=|\omega|$ $\mu$-a.e.\ (in particular, $\||\underline\omega|\|_{L^2(\mu)}=1$) and $P(\underline\omega)=\omega$. We  have
\[
\|\iota(v)\|_{L^2(\mu)}\geq \int \underline\omega(\iota(v))\,\d\mu=\int P(\underline\omega)(v)\,\d\mu=\int\omega(v)\,\d\mu\geq \|v\|_{L^2(\mu)}-\eps,
\]
which by the arbitrariness of $\eps$ and the inequality \eqref{eq:leq2} is sufficient to conclude.

The last claim is now obvious.
\end{proof}
\begin{remark}{\rm
In the proof of Theorem \ref{thm:isomemb} we did not use the fact the distance on $\R^d$ was the Euclidean one: the same conclusion holds by endowing it with the distance coming from any norm. 
}\fr\end{remark}

Now that we embedded $L^2_\mu(T\R^d)$ into $L^2(\R^d,\R^d;\mu)$ we can further proceed by associating to each vector field $v\in L^2_\mu(T\R^d)$ the current $\mathcal I(v)$ whose action on the smooth, compactly supported one form $\underline\omega$ is
\[
\la\mathcal I(v),\underline\omega\ra:=\int \underline\omega(\iota (v))\,\d\mu=\int P(\underline\omega)(v)\,\d\mu.
\]
It is clear that $\mathcal I(v)$ has locally finite mass and that, since $\iota$ preserves the pointwise norm, the mass measure $\|\mathcal I(v)\|$ is given by $|v|\mu$.

By definition, the boundary of $\mathcal I(v)$ acts on $f\in C^\infty_c(\R^d)$ as
\[
\la\partial\mathcal I(v),f\ra:=\la\mathcal I(v),\ud f\ra=\int \ud f(\iota (v))\,\d\mu=\int \d f(v)\,\d\mu.
\]
By looking at the third expression in this chain of equalities we see that $\partial\mathcal I(v)$ has locally finite mass (=is a Radon measure) if and only if the distributional divergence of $\iota(v)\mu$ is a measure and in this case such measure coincides with $-\partial\mathcal I(v)$. Looking at the fourth and last term, instead, and comparing it with Definition \ref{def:mdiv} we see the following:
\begin{corollary}\label{co:curr}
Let $v\in L^2_\mu(T\R^d)$ be with compact support. Then $\mathcal I(v)$ is a normal current if and only if $v\in D(\bold{div}_\mu)$ and in this case
\[
\partial\mathcal I(v)=-\bold{div}_\mu(v).
\]
\end{corollary}

\section{Statement and proof of the main result}\label{se:main}

Let us start collecting the known results we shall use. The first is a simple statement concerning the minimal weak upper gradient of the distance function. Here and in the following, given $x\in X$ we shall denote by $\sfd_x$ the function $y\mapsto\sfd(x,y)$.
\begin{proposition}
Let $(X,\sfd,\mm)$ be a $\RCD^*(K,N)$ space, $N<\infty$. Then for every $x\in X$ we have
\begin{equation}
\label{eq:norm1}
|\d\,\sfd_x|=1,\quad\mm-a.e..
\end{equation}
\end{proposition}
\begin{proof}
Recall that $\RCD^*(K,N)$ is doubling and supporting a 1-2 weak Poincar\'e inequality (\cite{Sturm06II},\cite{Rajala12}), that the local Lipschitz constant of $\sfd_x$ is identically 1 (because the space is geodesic) and conclude applying Cheeger's results in \cite{Cheeger00}.

An alternative argument which does not use the results in \cite{Cheeger00} but relies instead on the additional regularity of both the space and the function considered goes as follows. The function $\frac12\sfd_x^2$ is $c$-concave and thus a Kantorovich potential from any chosen measure $\mu_0$ and some measure $\mu_1$ depending on $\mu_0$. Picking $\mu_0\leq C\mm$ for some $C>0$ and with bounded support, by the results in \cite{GigliRajalaSturm13} we know that the only geodesic $(\mu_t)$ from $\mu_0$ to $\mu_1$ is such that $\mu_t\leq C'\mm$  for any $t\in[0,\frac12]$. Thus we can apply the metric Brenier theorem (see Theorem 10.3 in \cite{AmbrosioGigliSavare11}) to deduce that $|\d\frac{\sfd_x^2}{2}|$ coincides $\mm$-a.e.\ with the upper slope of $\frac{\sfd_x^2}{2}$. Since  $(X,\sfd,\mm)$ is doubling,  the upper slope coincides $\mm$-a.e.\ with the lower one (see Proposition 2.7 in \cite{AmbrosioGigliSavare11}) and being $(X,\sfd)$ geodesic, the latter is easily seen to be identically $\sfd_x$ by direct computation. Thus we know that $|\d\frac{\sfd_x^2}{2}|=\sfd_x$ $\mm$-a.e.\ and the conclusion follows from the chain rule.
\end{proof}
Next, we recall the main result of Mondino-Naber in the form we shall use, in particular making explicit some of the ingredients that we will need:
\begin{theorem}\label{thm:MN}
Let $(X,\sfd,\mm)$ be a $\RCD^*(K,N)$ space. Then there are disjoint Borel sets $A_i\subset X$, $i=1,\ldots,n$ with $n\leq N$ covering $\mm$-a.e.\ $X$ such that the following holds.

For every $i=1,\ldots, n$ and $\eps>0$ there is a countable disjoint collection $(U^\eps_{i,j})_{j\in\N}$ of Borel  subsets of $A_i$ covering $\mm$-a.e.\ $A_i$ and, for every $j\in\N$, points $x^\eps_{i,j,k}$ with $k=1,\ldots, i$, such that
\begin{equation}
\label{eq:scalp}
|\langle\nabla\sfd_{x^\eps_{i,j,k}},\nabla\sfd_{x^\eps_{i,j,k'}}\rangle|\leq \eps\qquad \mm-a.e.\ on\ U^\eps_{i,j},\qquad\forall k\neq k'
\end{equation}
and so that the map $\varphi^\eps_{i,j}:X\to\R^i$ given by $\varphi^\eps_{i,j}(x):=(\sfd_{x^\eps_{i,j,1}}(x),\ldots,\sfd_{x^\eps_{i,j,i}}(x))$ satisfies
\begin{equation}
\label{eq:bil}
\varphi^\eps_{i,j}\restr{U^\eps_{i,j}}:U^\eps_{i,j}\quad\to\quad \varphi^\eps_{i,j}(U^\eps_{i,j})\qquad \text{ is $(1+\eps)$-biLipschitz}.
\end{equation}
\end{theorem}
\begin{proof} This statement has been proved in \cite{Mondino-Naber14}, however, since some of the claims that we make only appear implicitly in the course of the various proofs, for completeness we point out where such claims appear.

The fact that $X$ can be covered by Borel charts $(1+\eps)$-biLipschitz to subsets of the Euclidean space is the main result in \cite{Mondino-Naber14}. The fact that the coordinates of the charts are distance functions is part of the construction, see \cite[Theorem 6.5]{Mondino-Naber14} (more precisely, in \cite{Mondino-Naber14} the coordinates are distance functions plus well chosen constants, so that 0 is always in the image, but this has no effect for our discussion). 

Thus we are left to prove \eqref{eq:scalp}. Looking at the construction of the sets $U^\eps_{i,j}$ in \cite{Mondino-Naber14} we see that they are contained in the set of $x$'s such that
\begin{equation}
\label{eq:MN}
\sup_{r'\in(0,r)}\frac1{\mm(B_{r'}(x))}\int_{B_{r'}(x)}\sum_{1\leq k\leq k'\leq i}\Big|\d\Big(\tfrac{\sfd_{x^\eps_{i,j,k}}+\sfd_{x^\eps_{i,j,k'}}}{\sqrt 2}-\sfd_{x^\eps_{i,j,k,k'}}\Big) \Big|^2\,\d\mm\leq \eps_1,
\end{equation}
where $r,\eps_1>0$ are bounded from above in terms of $K,N,\eps$ only and the points $x^\eps_{i,j,k,k'}$ are built together with the $x^\eps_{i,j,k}$'s (in \cite{Mondino-Naber14}  $x_{i,j,k},x_{i,j,k'},x_{i,j,k,k'}$ are called $p_i,p_j,p_i+p_j$ respectively). We remark that the choice of $r,\eps_1$ affects the construction of the sets $U_{i,j}^\eps$ and the points $x^\eps_{i,j,k}$, and that in any case  $\eps_1$ can be chosen to be smaller than $\big|\tfrac{\eps}{\sqrt 2+1}\big|^2$.

\bigskip

Notice that in \cite{Mondino-Naber14} the distance in \eqref{eq:MN} is scaled by a factor $r$, whose only effect is that $r'$ varies in $(0,1)$ rather than in $(0,r)$. The validity of \eqref{eq:MN} comes from the definition of maximal function, called $M^k$,  given in \cite[Equation/Definition (67)]{Mondino-Naber14}, the fact that the sets called $U^k_{\eps_1,\delta_1}$ introduced in \cite[Equation/Definition (70)]{Mondino-Naber14} are contained, by definition, in $\{M^k\leq \eps_1\}$ and the fact that the charts as given by \cite[Theorem 6.5]{Mondino-Naber14} are defined on the sets $B^{\tilde\sfd}_{\delta_1}\cap U^k_{\eps_1,\delta_1}\subset U^k_{\eps_1,\delta_1}$. Notice also that in \cite{Mondino-Naber14} the notion of weak upper gradient $|D f|$ of a function $f$ is used, in place of the pointwise norm of the differential used in our writing of \eqref{eq:MN}, but the two objects coincide (see \cite{Gigli14}).

\bigskip

We come back to the proof of \eqref{eq:scalp}. Recall that, being $\mm$ doubling (see \cite{Sturm06II}), Lebesgue differentiation theorem holds. Hence  from \eqref{eq:MN} and the discussion thereafter we see that up to a properly choosing $\eps_1$, and thus $U^\eps_{i,j},x^\eps_{i,j,k},x^\eps_{i,j,k,k'}$, we can assume  that
\[
\Big|\d\Big(\tfrac{\sfd_{x^\eps_{i,j,k}}+\sfd_{x^\eps_{i,j,k'}}}{\sqrt 2}-\sfd_{x^\eps_{i,j,k,k'}}\Big) \Big|^2\leq\big|\tfrac{\eps}{\sqrt 2+1}\big|^2\quad\mm-a.e.\ on\ U_{i,j}^\eps.
\]
Thus to conclude it is sufficient to prove that for given $x_1,x_2,y\in X$ we have
\[
\Big|\d\Big(\tfrac{\sfd_{x_1}+\sfd_{x_2}}{\sqrt 2}-\sfd_{y}\Big) \Big|\leq\tfrac{\eps}{\sqrt 2+1} \qquad\Rightarrow\qquad |\la\nabla\sfd_{x_1},\nabla\sfd_{x_2}\ra|\leq\eps.
\]
This follows with minor algebraic manipulations from the identity \eqref{eq:norm1}:
\[
\begin{split}
|\la\nabla\sfd_{x_1},\nabla\sfd_{x_2}\ra|&=\Big|\big|\d\big(\tfrac{\sfd_{x_1}+\sfd_{x_2}}{\sqrt 2}\big)\big|^2-\tfrac{|\d\sfd_{x_1}|^2+|\d\sfd_{x_2}|^2}{2}\Big|\\
&=\Big|\big|\d\big(\tfrac{\sfd_{x_1}+\sfd_{x_2}}{\sqrt 2}\big)\big|^2-1\Big|\\
&=\Big|\big|\d\big(\tfrac{\sfd_{x_1}+\sfd_{x_2}}{\sqrt 2}\big)\big|^2-|\d\sfd_y|^2\Big|\\
&=\Big|\la\d\big(\tfrac{\sfd_{x_1}+\sfd_{x_2}}{\sqrt 2}\big)+\d\sfd_y\,,\,\d\big(\tfrac{\sfd_{x_1}+\sfd_{x_2}}{\sqrt 2}\big)-\d\sfd_y\ra\Big|\\
&\leq(\sqrt 2+1)\big|\d\big(\tfrac{\sfd_{x_1}+\sfd_{x_2}}{\sqrt 2}\big)-\d\sfd_y\big|.
\end{split}
\]
\end{proof}

The last result we shall need is the Laplacian comparison estimate for the distance function obtained in \cite{Gigli12}. Such result holds in the sharp form, but we recall it in qualitative form, sufficient for our purposes:
\begin{theorem}\label{prop:distrlap}
Let $(X,\sfd,\mm)$ be a $\RCD^*(K,N)$ space and $x\in X$. Then the distributional Laplacian of $\sfd_x$ in $X\setminus\{x\}$ is a measure, i.e.\ there exists a Radon measure $\mu$ on $X$ such that for every $f\in \LIP_{\rm bs}(X)$ with $\supp(f)\subset X\setminus\{x\}$ it holds
\begin{equation}
\label{eq:distrlap}
\int\la\nabla f,\nabla\sfd_x\ra\,\d\mm=-\int f\,\d\mu.
\end{equation}
\end{theorem}
Read in terms of measure-valued divergence, the above theorem yields:
\begin{corollary}\label{cor:divm}
Let $(X,\sfd,\mm)$ be a $\RCD^*(K,N)$ space, $x\in X$ and $\psi\in\LIP(X)$ with support compact and contained in $X\setminus \{x\}$. Then the vector field $\psi\nabla\sfd_x\in L^2(TX)$ belongs to $D(\bold{div}_\mm)$, i.e.\ it has measure valued divergence on $X$ in the sense of Definition \ref{def:mdiv}.
\end{corollary}
\begin{proof}
Since $|\psi\nabla\sfd_x|\leq |\psi|$ it is clear that $\psi\nabla\sfd_x\in L^2(TX)$. Theorem \ref{prop:distrlap} above, the very definition of measure valued divergence given in \ref{def:mdiv} and the Leibniz rule given in Proposition \ref{prop:leibdiv} ensure that $\psi\nabla\sfd_x\in D(\bold{div}_\mm,X\setminus\{x\})$. On the other hand, by construction $\psi\nabla\sfd_x$ is 0 on a neighbourhood of $x$ and thus,  trivially, has 0 measure valued divergence in such neighbourhood. The conclusion comes from Proposition \ref{prop:localdiv} 
\end{proof}
We now have all the ingredients to prove our main result:
\begin{theorem}\label{thm:main}
With the same notations and assumptions of Theorem \ref{thm:MN}, we pick $\eps<\frac1N$.

Then for every $i,j$ we have
\[
(\varphi^\eps_{i,j})_*(\mm\restr{U^\eps_{i,j}})\ll\mathcal L^i.
\]
\end{theorem}
\begin{proof} 

\noindent{\bf Set up} By the inner regularity of $\mm$ applied to the sets $U^\eps_{i,j}\setminus\{x^\eps_{i,j,1},\ldots,x^\eps_{i,j,i}\}$ we can assume that the $U^\eps_{i,j}$'s are compact and that $x^\eps_{i,j,k}\notin U^\eps_{i,j}$ for every $k=1,\ldots,i$. Now fix $i,j$ and, for brevity, write $\varphi,U, x_1,\ldots,x_i$ in place of $\varphi^\eps_{i,j},U^\eps_{i,j},x^\eps_{i,j,1},\ldots,x^\eps_{i,j,i}$  respectively.

\noindent{\bf Step 1: Normal currents}
Let $(\psi^\delta)_{\delta>0}$ be a family of Lipschitz, compactly supported $[0,1]$-valued maps on $X$ pointwise converging to $\nchi_U$ as $\delta\downarrow0$ and consider the vector fields
\[
v^\delta_k:=\psi^\delta\nabla\sfd_{x_{k}}\in L^2(TX),\qquad\forall k=1,\ldots,i.
\]
By Corollary \ref{cor:divm} we know that $v^\delta_k\in D(\bold{div}_\mm)$.

Now observe that $\varphi:X\to\R^i$ is a Lipschitz and proper map (i.e.\ the preimage of compact sets is compact) and thus $\mu:=\varphi_*\mm$ is a Radon measure on $\R^i$ and equipping $\R^i$ with such measure we see that $\varphi:X\to\R^i$ is of bounded deformation. By Proposition \ref{prop:pfdiv} the vector fields
\[
u^\delta_k:={\sf Pr}_{\varphi }(\d\varphi (v^\delta_k))\in L^2_\mu(T\R^i)
\] 
all belong to $D(\bold{div}_\mu)$ and since by construction they have compact support we see from Corollary \ref{co:curr} that the currents 
\[
\mathcal I(u^\delta_k)=\overrightarrow{\mathcal I(u^\delta_k)}\|{\mathcal I(u^\delta_k)}\|=\frac{\iota(u^\delta_k)}{|u^\delta_k|}(|u^\delta_k|\mu),\qquad k=1,\ldots,i
\] 
are normal. We also notice that trivially 
\begin{equation}
\label{eq:triv}
\mu\restr{\{|u^\delta_k|>0\}}\ll |u^\delta_k|\mu= \|{\mathcal I(u^\delta_k)}\|,\qquad k=1,\ldots,i.
\end{equation}

\noindent{\bf Step 2: Independent vector fields}
We  claim that 
\begin{equation}
\label{eq:inddist}
\text{the vector fields $\nabla\sfd_{x_1},\ldots,\nabla\sfd_{x_i}\in L^2_{loc}(TX)$ are independent on $U$}
\end{equation} 
and to prove this we shall use our choice of $\eps<\frac1N$.

Let $f_1,\ldots,f_i\in L^\infty(X)$ be such that $\sum_{k=1}^if_k\nabla\sfd_{x_k}=0$ $\mm$-a.e.\ on $U$ and notice that
\[
0=\langle\nabla\sfd_{x_k},\sum_{k'=1}^if_{k'}\nabla\sfd_{x_{k'}}\rangle=f_k|\d\sfd_{x_k}|^2+\sum_{k'\neq k}f_{k'}\langle\nabla\sfd_{x_k},\nabla\sfd_{x_{k'}}\rangle\quad\mm-a.e.\ on\ U.
\]
From \eqref{eq:norm1},  \eqref{eq:scalp} and the fact that $\eps<\frac1N$ we obtain 
\[
|f_k|=|f_k||\d\sfd_{x_k}|^2\leq \sum_{k'\neq k}|f_{k'}|\,|\langle\nabla\sfd_{x_k},\nabla\sfd_{x_{k'}}\rangle|\leq \frac1N \sum_{k'\neq k}|f_{k'}|\qquad\mm-a.e.\ on\ U.
\]
Adding up in $k=1,\ldots,i$  we deduce $\sum_k|f_k|\leq \frac {i-1}N\sum_k|f_k|$ $\mm$-a.e.\ on $U$, and since  by Theorem \ref{thm:MN} we know that $i\leq N$, this  forces $\sum_k|f_k|=0$ $\mm$-a.e.\ on $U$, which is the claim \eqref{eq:inddist}.

Now notice that  Theorem \ref{thm:MN} grants that $\varphi:X\to \R^i$ is of bounded deformation (having equipped $\R^i$ with the measure $\mu=\varphi_*\mm$), partially invertible on $U$ and such that $(\varphi\restr U)^{-1}$ is Lipschitz. Therefore Proposition \ref{prop:ind} grants that the vector fields
\[
u^0_k:=\P_\varphi(\d\varphi(\nchi_U\nabla\sfd_{x_k}))\in L^2_\mu(T\R^i) \qquad k=1,\ldots,i
\]
are independent on $\{\P_\varphi(\nchi_U)>0\}$, which by Proposition \ref{thm:isomemb} is the same as to say that 
\begin{equation}
\label{eq:ind}
\text{$\iota(u_1^0)(x),\ldots,\iota(u^0_i)(x)\in\R^i$ are independent for $\mu$-a.e.\ $x$ such that $\P_\varphi(\nchi_U)(x)>0$.
}
\end{equation}

\noindent{\bf Conclusion} The fact that  the family $(\psi^\delta)$ is equibounded in $L^\infty(X)$ and pointwise converges to $\nchi_U$  easily implies that $v^\delta_k\to \nchi_U\nabla\sfd_{x_k}$ in $L^2(TX)$ for every $k=1,\ldots,i$. By the continuity of $\d\varphi$ and $\P_\varphi$ we deduce that for any $k=1,\ldots,i$  we have $u^\delta_k\to u^0_k$ in $L^2_\mu(T\R^i)$ as $\delta\downarrow0$ and thus by Proposition \ref{thm:isomemb} that $\iota(u^\delta_k)\to \iota(u^0_k)$ in $L^2(\R^i,\R^i;\mu)$ as $\delta\downarrow0$.

Let $A^\delta\subset \R^i$, $\delta\geq0$ be the Borel sets defined, up to $\mu$-negligible sets, as
\[
A^\delta:=\Big\{x\in\R^i\ :\ \iota(u^\delta_1)(x),\ldots,\iota(u^\delta_i)(x)\text{ are independent}\Big\}.
\]
Since being an independent family is an open condition, the convergence just proved ensures that for any $\delta_n\downarrow0$ we have
\begin{equation}
\label{eq:a0}
\mu\Big(A^0\setminus\bigcup_nA^{\delta_n}\Big)=0.
\end{equation}
For $\delta>0$, we apply Theorem \ref{thm:DPR} to the currents $\mathcal I(u^\delta_k)$: since $\mu$-a.e.\ on $A^\delta$ the vectors $\iota(u^\delta_k)$ are all nonzero, we have  $\mu\restr{A^\delta}\ll \mu\restr{\{|u^\delta_k|>0\}}$ for every $k=1,\ldots,i$ and thus \eqref{eq:triv} and Theorem \ref{thm:DPR} grant that
\[
\mu\restr {A^{\delta_n}}\ll\mathcal L^i\qquad \forall n\in\N
\]
and thus from \eqref{eq:a0} we deduce
\[
\mu\restr{A^0}\ll\mathcal L^i.
\]
On the other hand by \eqref{eq:ind} we know that, up to $\mu$-negligible sets, we have $A_0\supset \{\P_\varphi(\nchi_U)>0\}$ which together with the above implies
\[
\P_\varphi(\nchi_U)\mu\ll\mathcal L^i.
\]
As we have $\varphi_*(\mm\restr U)=\P_\varphi(\nchi_U)\mu$, the proof is achieved.
\end{proof}

\def\cprime{$'$} \def\cprime{$'$}

\end{document}